\documentclass[11pt]{amsart}
\usepackage{latexsym}
\usepackage{fullpage}
\usepackage{amssymb}
\usepackage{amscd}
\usepackage{float}
\usepackage{graphicx}
\usepackage{amsfonts}
\usepackage{pb-diagram}
 \usepackage{amsmath,amscd}

\newtheorem{thm}{Theorem}

\newtheorem{prop}{Proposition}
\newtheorem{defn}{Definition}
\newtheorem{remark}{Remark}

\newtheorem{nt}{Notation}

\begin{document}

\title[The Kauffman bracket skein module of the handlebody of genus 2 via braids]
  {The Kauffman bracket skein module of the handlebody of genus 2 via braids}

\author{Ioannis Diamantis}
\address{ International College Beijing,
China Agricultural University,
No.17 Qinghua East Road, Haidian District,
Beijing, {100083}, P. R. China.}
\email{ioannis.diamantis@hotmail.com}



\setcounter{section}{-1}

\date{}

\begin{abstract}
In this paper we present two new bases, $B^{\prime}_{H_2}$ and $\mathcal{B}_{H_2}$, for the Kauffman bracket skein module of the handlebody of genus 2 $H_2$, KBSM($H_2$). We start from the well-known Przytycki-basis of KBSM($H_2$), $B_{H_2}$, and using the technique of parting we present elements in $B_{H_2}$ in open braid form. We define an ordering relation on an augmented set $L$ consisting of monomials of all different ``loopings'' in $H_2$, that contains the sets $B_{H_2}$, $B^{\prime}_{H_2}$ and $\mathcal{B}_{H_2}$ as proper subsets. Using the Kauffman bracket skein relation we relate $B_{H_2}$ to the sets $B^{\prime}_{H_2}$ and $\mathcal{B}_{H_2}$ via a lower triangular infinite matrix with invertible elements in the diagonal. The basis $B^{\prime}_{H_2}$ is an intermediate step in order to reach at elements in $\mathcal{B}_{H_2}$ that have no crossings on the level of braids, and in that sense, $\mathcal{B}_{H_2}$ is a more natural basis of KBSM($H_2$). Moreover, this basis is appropriate in order to compute Kauffman bracket skein modules of c.c.o. 3-manifolds $M$ that are obtained from $H_2$ by surgery, since isotopy moves in $M$ are naturally described by elements in $\mathcal{B}_{H_2}$.
\end{abstract}

\maketitle

\section{Introduction and overview}\label{intro}

Skein modules were independently introduced by Przytycki \cite{P} and Turaev \cite{Tu} as generalizations of knot polynomials in $S^3$ to knot polynomials in arbitrary 3-manifolds. The essence is that skein modules are quotients of free modules over ambient isotopy classes of links in 3-manifolds by properly chosen local (skein) relations.

\begin{defn}\rm
Let $M$ be an oriented $3$-manifold and $\mathcal{L}_{{\rm fr}}$ be the set of isotopy classes of unoriented framed links in $M$. Let $R=\mathbb{Z}[A^{\pm1}]$ be the Laurent polynomials in $A$ and let $R\mathcal{L}_{{\rm fr}}$ be the free $R$-module generated by $\mathcal{L}_{{\rm fr}}$. Let $\mathcal{S}$ be the ideal generated by the skein expressions $L-AL_{\infty}-A^{-1}L_{0}$ and $L \bigsqcup {\rm O} - (-A^2-A^{-1})L$, where $L_{\infty}$ and $L_{0}$ are represented schematically by the illustrations in Figure~\ref{skein}. Note that blackboard framing is assumed. 

\begin{figure}[!ht]
\begin{center}
\includegraphics[width=1.9in]{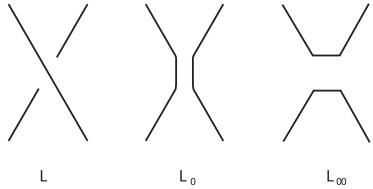}
\end{center}
\caption{The links $L_{\infty}$ and $L_{0}$ locally.}
\label{skein}
\end{figure}

\noindent Then the {\it Kauffman bracket skein module} of $M$, KBSM$(M)$, is defined to be:

\begin{equation*}
{\rm KBSM} \left(M\right)={\raise0.7ex\hbox{$
R\mathcal{L} $}\!\mathord{\left/ {\vphantom {R\mathcal{L_{{\rm fr}}} {\mathcal{S} }}} \right. \kern-\nulldelimiterspace}\!\lower0.7ex\hbox{$ S  $}}.
\end{equation*}

\end{defn}

In \cite{P} the Kauffman bracket skein module of the handlebody of genus 2, $H_2$, is computed using diagrammatic methods by means of the following theorem:

\begin{thm}[\cite{P}]\label{tprz}
The Kauffman bracket skein module of $H_2$, KBSM($H_2$), is freely generated by an infinite set of generators $\left\{x^i\, y^j\, z^k,\ (i, j, k)\in \mathbb{N}\times \mathbb{N}\times \mathbb{N}\right\}$, where $x^i\, y^j\, z^k$ is shown in Figure~\ref{przt}.
\end{thm}

\noindent Note that in \cite{P}, $H_2$ is presented by a twice punctured plane and diagrams in $H_2$ determine a framing of links in $H_2$. The two dots in Figure~\ref{przt} represent the punctures.

\begin{figure}[!ht]
\begin{center}
\includegraphics[width=2.6in]{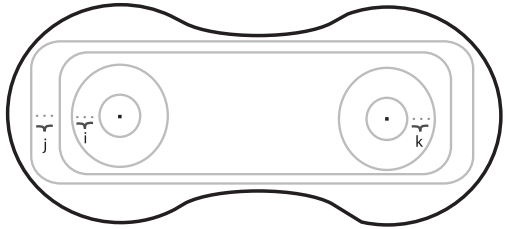}
\end{center}
\caption{The Przytyzki-basis of KBSM($H_2$), $B_{H_2}$.}
\label{przt}
\end{figure}

 It is worth mentioning that in \cite{KL} the authors define the mixed Hecke algebra $H_{2, n}(q)$, which is related to the knot theory of the handlebody of genus 2 and to other families of 3-manifolds, and provide a spanning set as well as a potential linear basis of $H_{2, n}(q)$. Their motivation is the computation of HOMFLYPT skein modules of families of 3-manifolds via braids. Note also that in \cite{DL2, DLP, DL3, DL4}, recent results on the computation of the HOMFLYPT skein module of the lens spaces $L(p, 1)$ via braids are presented with the use of the {\it generalized Iwahori-Hecke algebra of type B}. Finally, in \cite{GM} the authors present a new basis for the Kauffman bracket skein module of the solid torus diagrammatically, and the same basis was recovered in \cite{D} with the use of the {\it Tempereley-Lieb algebra of type B}. 

\bigbreak

In this paper we present two new bases for KBSM($H_2$) with the use of braids and techniques developed in \cite{LR1, LR2, La1,OL, DL2, D}. More precisely, we start from the Przytycki-basis $B_{H_2}$ of KBSM($H_2$) and using Alexander's theorem in $H_2$ and the looping generators illustrated in Figure~\ref{Loops} (see also Definition~\ref{loopings}) we present elements in $B_{H_2}$ to open braid form. Using the technique of parting introduced in \cite{LR1} we have that braids in $H_2$ form a group, $B_{2, n}$. We then introduce the sets $B^{\prime}_{H_2}$ (see Eq.~\ref{nbasis1}) and $\mathcal{B}_{H_2}$ (see Eq.~\ref{nbasis2}) and using the ordering relation defined in Definition~\ref{order}, we show that these sets form bases for KBSM($H_2$). 

\smallbreak

The main results of this paper are the following:

\begin{thm}\label{newbasis1}
The following set forms a basis for KBSM($H_2$):
\begin{equation}\label{nbasis1}
B^{\prime}_{H_2}\ =\ \{t^i\, {\tau^{\prime}_1}^k\, {T^{\prime}_2}^j,\ i, j, k \in \mathbb{N} \}.
\end{equation}
\end{thm}

\smallbreak

The method for proving Theorem~\ref{newbasis1} is the following:

\smallbreak

\begin{itemize}
\item[$\bullet$] we first recall the total ordering defined in \cite{DL2} and extend it to monomials of $t^{\prime}_i$'s, $\tau^{\prime}_k$'s and $T^{\prime}_j$'s.
\smallbreak
\item[$\bullet$] We perform the technique of parting and separate the loop generators into three sets, each one independent from the rest. 
\smallbreak
\item[$\bullet$] Applying the Kauffman bracket skein relation we express elements in the $B_{H_2}$ basis to elements in the set $B^{\prime}_{H_2}$ in a unique way, and we relate $B_{H_2}$ to $B^{\prime}_{H_2}$ by a lower triangular matrix with invertible elements in the diagonal. The result follows.
\end{itemize}

\bigbreak

We then move toward a different basic set for KBSM($H_2$). More precisely we have:

\begin{thm}\label{newbasis2}
The following set forms a basis for KBSM($H_2$):
\begin{equation}\label{nbasis2}
\mathcal{B}_{H_2}\ =\ \{t^i\, {\tau^{\prime}}^k\, {T^{\prime}}^j,\ i, j, k \in \mathbb{N} \}.
\end{equation}
\end{thm}

\smallbreak

The method for proving Theorem~\ref{newbasis2} is similar to the one followed in proving Theorem~\ref{newbasis1}. In particular, we start from elements in the new basis $B^{\prime}_{H_2}$ in algebraic mixed braid form and we apply the Kauffman bracket skein relation on the braid crossings. In that way all crossings on the braid level are canceled and elements in the set $\mathcal{B}_{H_2}$ are obtained.

\smallbreak

The paper is organized as follows: In \S\ref{basics} we recall the setting and the essential techniques and results from \cite{OL, La1, LR1, LR2, DL1}. More precisely, we present isotopy moves for knots and links in $H_2$ and we then describe braid equivalence for knots and links in $H_2$. In \S\ref{H_2} we present the Przytycki-basis of KBSM($H_2$) in open braid form. In \S~2 we define the augmented set $L$, consisting of all possible monomials in the looping generators, and in \S~2.1 we introduce an ordering relation on $L$. We show that the set $L$ equipped with this ordering relation is a totally ordered (Proposition~\ref{totord}) and well-ordered (Proposition~\ref{totord2}) set and, using the notion of homologous elements (see Definition~\ref{hom}), we prove that the sets $B^{\prime}_{H_2}$ and $\mathcal{B}_{H_2}$ form bases for KBSM($H_2$).

\bigbreak

\noindent \textbf{Acknowledgments}\ \ I would like to acknowledge several discussions with Ms. Marianna Press. Moreover, financial support by the International College Beijing, China Agricultural University is gratefully acknowledged.

\section{Preliminaries}\label{basics}

\subsection{Mixed links and isotopy in $H_2$}

We consider $H_2$ to be

$$S^3\backslash \{{\rm open\ tubular\ neighborhood\ of}\ I_2 \},$$

\smallbreak

\noindent where $I_2$ denotes the point-wise fixed identity braid on two indefinitely extended strands meeting at the point at infinity. Thus $H_2$ may be represented in $S^3$ by the braid $I_2$. As explained in \cite{LR1, LR2, La1, OL, DL1}, an oriented link $L$ in $H_2$ can be represented by an oriented \textit{mixed link} in $S^{3}$, that is a link in $S^{3}$ consisting of the fixed part $\widehat{I_2}$ and the moving part $L$ that links with $\widehat{I_2}$ (see Figure~\ref{mlnk}). A \textit{mixed link diagram} is a diagram $\widehat{I_2}\cup \widetilde{L}$ of $\widehat{I_2}\cup L$ on the plane of $\widehat{I_2}$, where this plane is equipped with the top-to-bottom direction of $I_2$.

\begin{figure}[!ht]
\begin{center}
\includegraphics[width=1.9in]{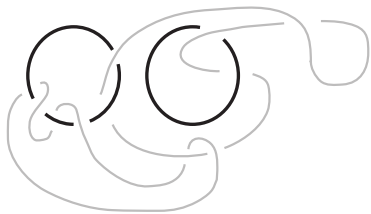}
\end{center}
\caption{A mixed link.}
\label{mlnk}
\end{figure}

We now translate isotopy of an oriented link $L$ in $H_2$ to isotopy of its corresponding mixed link in $S^{3}$. Mixed link isotopy consists of a sequence of moves that keep the oriented $\widehat{I_2}$ point-wise fixed, that is, isotopy in $S^{3}$ together with the \textit{mixed Reidemeister moves}. Note that $L$ will avoid the 2 hollow tubes of $H_2$, and also it will not pass beyond the boundary of $H_2$ from either end. In terms of diagrams we have the following:

\smallbreak

The mixed link equivalence in $S^{3}$ includes the classical Reidemeister moves and the mixed Reidemeister moves, which involve the fixed and the standard part of the mixed link, keeping $\widehat{I_2}$ pointwise fixed (cf. Thm.~5.5 \cite{LR1} \& Thm.~1 \cite{OL}).

\subsection{Mixed braids and braid equivalence for knots and links in $H_2$}

In this subsection we translate isotopy of links in $H_2$ into braid equivalence. We need to introduce the notion of a geometric mixed braid first. A \textit{geometric mixed braid} related to $H_2$ and to a link $L$ in $H_2$ is an element of the group $B_{2+n}$, where $2$ strands form the fixed \textit{identity braid} $I_2$ representing $H_2$ and $n$ strands form the {\it moving subbraid\/} $\beta$ representing the link $L$ in $H_2$. For an illustration see the middle picture of Figure~\ref{mbraid}. By the Alexander theorem for knots and links in $H_2$ (cf. Thm.~5.4 \cite{LR1} \& Thm.~2 \cite{OL}), a mixed link diagram $\widehat{I_2}\cup \widetilde{L}$ of $\widehat{I_2}\cup L$ may be turned into a \textit{geometric mixed braid} $I_2\cup \beta$ with isotopic closure.

\smallbreak

We further need the notion of the $L$-moves (\cite[Definition 2.1]{LR1}). An {\it $L$-move} on a geometric mixed braid $I_2 \bigcup \beta$,
consists in cutting an arc of the moving subbraid $\beta$ open and pulling the upper cutpoint downward and
the lower  upward, so as to create a new pair of braid strands with corresponding endpoints
(on the vertical line of the cutpoint), and such that both strands  cross entirely  {\it over} or {\it under}
with the rest of the braid. Stretching the new strands over  will give rise to an {\it $L_o$-move\/} and under to an {\it  $L_u$-move\/}.
For an illustration see Figure~\ref{mbraid}.

\begin{figure}[!ht]
\begin{center}
\includegraphics[width=3.7in]{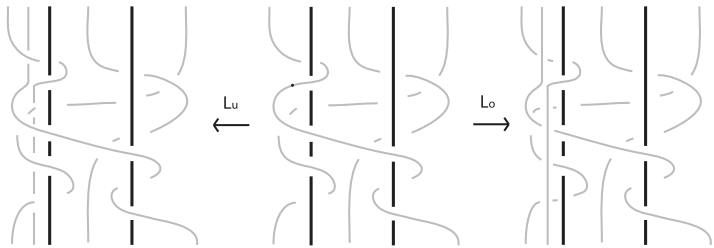}
\end{center}
\caption{A geometric mixed braid and the two types of $L$-moves.}
\label{mbraid}
\end{figure}

Two geometric mixed braids shall be called {\it $L$-equivalent} if and only if they differ by a sequence of $L$-moves and braid isotopy. Note that an $L$-move does not touch the fixed subbraid $I_2$.

\begin{thm}[Geometric braid equivalence in $H_2$, Theorem 3 \cite{OL}] \label{geommarkov}
Two oriented links in $H_2$ are isotopic if and only if any two corresponding geometric mixed braids of theirs in $S^3$ differ by a finite sequence of $L$-moves and isotopies of geometric mixed braids.
\end{thm}

\subsection{Algebraic mixed braids}

We now pass from the geometric mixed braid equivalence to an algebraic statement for links in $H_2$. We first perform the technique of {\it parting} introduced in \cite{LR2} in order to separate the moving strands from the fixed strands that represents $H_2$. This can be realized by pulling each pair of corresponding moving strands to the right and {\it over\/} or {\it under\/} the fixed strand that lies on their right. More precisely, we start from the rightmost pair respecting the position of the endpoints. The result of parting is a {\it parted mixed braid}. If the strands are pulled always over the strands of $I_2$, then this parting is called {\it standard parting}. See Figure~\ref{parting} for the standard parting of an abstract mixed braid. We will also call parting the converse operation, where a parted mixed braid is turned to a geometric mixed braid again.

\begin{figure}[!ht]
\begin{center}
\includegraphics[width=2.5in]{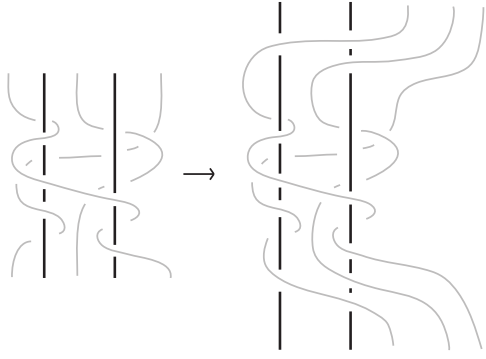}
\end{center}
\caption{ Parting a geometric mixed braid. }
\label{parting}
\end{figure}

An \textit{algebraic mixed braid} is a mixed braid on $2+n$ strands such that the first $2$ strands are fixed and form the identity braid on $2$ strands $I_2$ and the next $n$ strands are moving strands and represent a link in $H_2$. The set of all algebraic mixed braids on $2+n$ strands forms a subgroup of $B_{2+n}$, denoted $B_{2,n}$, called the {\it mixed braid group}. The mixed braid group $B_{2,n}$ has been introduced and studied in \cite{La1} (see also \cite{OL}) and it is shown that it has the presentations:

\begin{center} \label{B}
$
B_{2,n} = \left< \begin{array}{ll}  \begin{array}{l}
t, \tau,  \\
\sigma_1, \ldots ,\sigma_{n-1}  \\
\end{array} &
\left|
\begin{array}{l} \sigma_k \sigma_j=\sigma_j \sigma_k, \ \ |k-j|>1   \\
\sigma_k \sigma_{k+1} \sigma_k = \sigma_{k+1} \sigma_k \sigma_{k+1}, \ \  1 \leq k \leq n-1  \\
t \sigma_k = \sigma_k t, \ \ k \geq 2   \\
\tau \sigma_k = \sigma_k \tau, \ \ k \geq 2    \\
 \tau \sigma_1 \tau \sigma_1 = \sigma_1 \tau \sigma_1 \tau \\
 t \sigma_1 t \sigma_1 = \sigma_1 t \sigma_1 t \\
 \tau (\sigma_1 t {\sigma^{-1}_1}) =  (\sigma_1 t {\sigma^{-1}_1})  \tau
\end{array} \right.  \end{array} \right>,
$
\end{center}

or

\begin{center} \label{B2}
$
B_{2,n} = \left< \begin{array}{ll}  \begin{array}{l}
t, T,  \\
\sigma_1, \ldots ,\sigma_{n-1}  \\
\end{array} &
\left|
\begin{array}{l} \sigma_k \sigma_j=\sigma_j \sigma_k, \ \ |k-j|>1   \\
\sigma_k \sigma_{k+1} \sigma_k = \sigma_{k+1} \sigma_k \sigma_{k+1}, \ \  1 \leq k \leq n-1  \\
t \sigma_k = \sigma_k t, \ \ k \geq 2    \\
T \sigma_k = \sigma_k T, \ \ k \geq 2    \\
 t \sigma_1 T \sigma_1 = \sigma_1 T \sigma_1 t  \\
\end{array} \right.  \end{array} \right>,
$
\end{center}

\noindent where the {\it loop generators} $t, \tau, T$ are as illustrated in Figure~\ref{Loops}.

\begin{figure}[!ht]
\begin{center}
\includegraphics[width=3.1in]{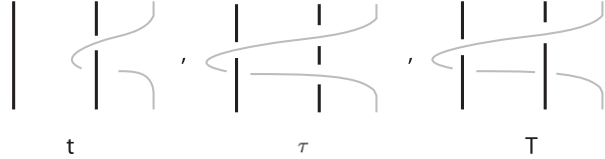}
\end{center}
\caption{ The loop generators $t, \tau$ and $T$ of $B_{2,n}$. }
\label{Loops}
\end{figure}

Let now $\mathcal{L}$ denote the set of oriented knots and links in $H_2$. Then, isotopy in $H_2$ is translated on the level of algebraic mixed braids by means of the following theorem:

\begin{thm}[Theorem~4, \cite{OL}] \label{markov}
 Let $L_{1} ,L_{2}$ be two oriented links in $H_2$ and let $I_2\cup \beta_{1} ,{\rm \; }I_2\cup \beta_{2}$ be two corresponding algebraic mixed braids in $S^{3}$. Then $L_{1}$ is isotopic to $L_{2}$ in $H_2$ if and only if $I_2\cup \beta_{1}$ is equivalent to $I_2\cup \beta_{2}$ by a finite sequence of L-moves and the braid relations in $\underset{n=1}{\overset{\infty}{\cup}} B_{2, n}$.
\end{thm}
\

\subsection{The Kauffman bracket skein module of $H_2$ via braids}\label{H_2}

In this subsection we present elements in the Przytyzki-basis of the Kauffman bracket skein module of $H_2$ in open braid form. For that we need the following:

\begin{defn}\label{loopings}\rm
We define the following elements in $B_{2, n}$:
\begin{equation}\label{algloops}
t_i^{\prime}\ :=\ \sigma_i\ldots \sigma_1\, t\, \sigma^{-1}_1 \ldots \sigma^{-1}_i,\ \ \tau_k^{\prime}\ :=\ \sigma_k\ldots \sigma_1\, \tau\, \sigma^{-1}_1 \ldots \sigma^{-1}_k,\ \ T_j^{\prime}\ :=\ \sigma_j\ldots \sigma_1\, T\, \sigma^{-1}_1 \ldots \sigma^{-1}_j
\end{equation}
\end{defn}

Note that the $t_i^{\prime}$'s, the $\tau_k^{\prime}$'s and the $T_j^{\prime}$'s are conjugate but not commuting. In order to simplify the algebraic expressions obtained throughout this paper, we introduce the following notation:

\begin{nt}\label{nt1}\rm
For $i, j\in \mathbb{N}$ such that $i<j$, we set:

\[
t^{a_i, a_j}_{i, j} \ := \ {t_i^{\prime}}^{a_i}\, {t_{i+1}^{\prime}}^{a_{i+1}}\, \ldots\, {t_j^{\prime}}^{a_j},\quad \tau^{b_{i, j}}_{i, j} \ := \ {\tau_i^{\prime}}^{b_i}\, {\tau_{i+1}^{\prime}}^{b_{i+1}}\, \ldots \, {\tau_j^{\prime}}^{b_j},\quad T^{c_{i, j}}_{i, j} \ := \ {T_i^{\prime}}^{c_i}\, {T_{i+1}^{\prime}}^{c_{i+1}}\, \ldots\, {T_j^{\prime}}^{c_j},
\]

\noindent where $a_k$, $b_k$ and $c_k \in \mathbb{N}, \forall k$, and

\begin{equation}
t_{i, j} \ := \ t_i^{\prime}\, t_{i+1}^{\prime}\, \ldots\, t_j^{\prime},\quad \tau_{i, j} \ := \ \tau_i^{\prime}\, \tau_{i+1}^{\prime}\, \ldots \, \tau_j^{\prime},\quad T_{i, j} \ := \ T_i^{\prime}\, T_{i+1}^{\prime}\, \ldots\, T_j^{\prime}.
\end{equation}

\noindent We also set:
$$t^{a_i, a_j}_{i, j}\, \tau^{b_{k, l}}_{k, l}\, T^{c_{m, n}}_{m, n}\, :=\, {t_i^{\prime}}^{a_i}\, \ldots \, {t_j^{\prime}}^{a_j}\, {\tau_{k}^{\prime}}^{b_{k}}\, \ldots \, {\tau_l^{\prime}}^{b_l}\, {T_{m}^{\prime}}^{c_{m}}\, \ldots\, {T_n^{\prime}}^{c_n},$$
\noindent where $i, j, k, l, m, n\in \mathbb{N}$ such that $i<j<k<l<m<n$.

\smallbreak

\noindent Finally, we note that ${t_i^{\prime}}^0\, {\tau_k^{\prime}}^0\, {T_j^{\prime}}^0$ represents the unknot.
\end{nt}

\smallbreak

We observe now that an element $x^iy^jz^k$ in the basis of KBSM($H_2$) described in Theorem~\ref{tprz} can be illustrated equivalently as a mixed link in $S^3$ so that we correspond the element $x^iy^jz^k$ to the minimal mixed braid representation in which we group the twists around each fixed strand. Figure~\ref{prbr} illustrates this correspondence. 

\begin{figure}[!ht]
\begin{center}
\includegraphics[width=4.1in]{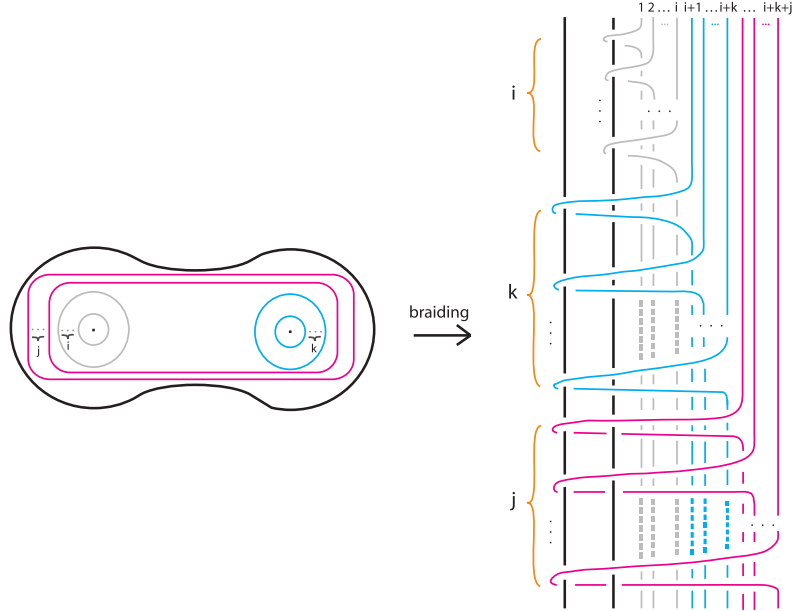}
\end{center}
\caption{ The Przytycki-basis in braid form. }
\label{prbr}
\end{figure}

Then, the set

$$B_{H_2}\, :=\, \{(tt^{\prime}_1\ldots t^{\prime}_i)\, (\tau^{\prime}_{i+1}\tau^{\prime}_{i+2}\ldots \tau^{\prime}_{i+k})\, (T^{\prime}_{i+k+1}T^{\prime}_{i+k+2}\ldots T^{\prime}_{i+k+j})\}, \ (i, k, j)\in \mathbb{N}\times \mathbb{N}\times \mathbb{N}$$

\noindent is a basis of KBSM($H_2$) in braid form. Equivalently, using Notation~\ref{nt1} we have that

$$B_{H_2}\ :=\ \{t_{0, i}\, \tau_{i+1, i+k}\, T_{i+k+1, i+k+j} \}, \ (i, k, j)\in \mathbb{N}\times \mathbb{N}\times \mathbb{N}.$$

\begin{remark}\label{sets}\rm
When expressing an element in the Przytycki-basis in open braid form, we obtain three different sets of strands: one set consists of the loopings $t^{\prime}_i$'s, the other set consists of the loopings $\tau^{\prime}_k$'s and the last one consists of the loopings $T^{\prime}_j$'s.
\end{remark}

\section{Two different bases for KBSM($H_2$)}

In this section we prove Theorems~\ref{newbasis1} and \ref{newbasis2}. For that we need to introduce the augmented set $L$, consisting of monomials in the $t_i^{\prime}$'s, $\tau_k^{\prime}$'s and $T_j^{\prime}$'s with exponents in $\mathbb{N}\cup \{0\}$ and  the notion of the index of a word in $L$. 

\begin{defn}\rm
Let $i_1, i_2, k_1, k_2, j_1, j_2 \in \mathbb{N}\cup \{0\}\, :\, i_1<i_2<k_1<k_2<j_1<j_2\ {\rm and}\ a_i, b_k, c_j \in \mathbb{N}\cup \{0\}\ \forall i, k, j$. We define the set $L$ to be:
\[
L\ :=\ \left\{t_{i_1, i_2}^{a_{i_1, i_2}}\, \tau_{k_1, k_2}^{b_{k_1, k_2}}\, T_{j_1, j_2}^{c_{j_1, j_2}}\right\}.
\]
\end{defn}

It is worth mentioning here that the sets $B_{H_2}$, $B^{\prime}_{H_2}$ and $\mathcal{B}_{H_2}$ are subsets of $L$. We now introduce the notions of index, length and sum of exponents of a monomial in $L$.

\begin{defn}\rm
\begin{itemize}
\item[i.] The $t$-index of a monomial $w$ in $L$, denoted $ind_t(w)$, is defined to be the highest index of the $t^{\prime}$'s in $w$. Similarly, the $\tau$-index, $ind_{\tau}(w)$, is defined to the highest index of the $\tau^{\prime}$'s in $w$ and the $T$-index, $ind_{T}(w)$, is defined to the highest index of the $T^{\prime}$'s in $w$.

\smallbreak



\item[ii.] We define the $t$-length of $w$, denoted by $l_t$, to be the difference $(i_2-i_1)$. Similarly, the $\tau$-length of $w$, denoted by $l_{\tau}$, is defined to be the difference $(k_2-k_1)$ and the $T$-length of $w$, denoted by $l_T$, is defined to be the difference $(j_2-j_1)$.

\smallbreak

\item[iii.] We finally define $exp_t(w)$ to be the sum of the exponents in the $t^{\prime}_i$'s in $w$. Similarly, $exp_{\tau}(w)$ is defined to be the sum of the exponents in the $\tau^{\prime}_i$'s in $w$ and $exp_T(w)$ is defined to be the sum of the exponents in the $T^{\prime}_i$'s in $w$. Finally, we define $exp(w)$ to be the sum of all exponents in $w$.
\end{itemize}
\end{defn}

\subsection{An ordering in the basis of $B_{H_2}$}

In \cite{DL2} an ordering relation is defined on elements in the bases of the HOMFLYPT skein module of the solid torus. We extend this ordering relation here to the set $L$ consisting of all monomials in the $t^{\prime}_i$'s, $\tau^{\prime}_k$'s and $T_j^{\prime}$'s. This ordering relation passes naturally to the basis $B_{H_2}$ of the Kauffman bracket skein module of $H_2$.

\begin{defn}\label{order}\rm
Let $a_i , b_j, c_k, d_l, e_m, f_n \in \mathbb{N}$ for all $i, j, k, l, m, n$ and:

\[
w=t_{i_1, i_2}^{a_{i_1}, a_{i_2}}\, \tau_{j_1, j_2}^{b_{j_1}, b_{j_2}}\, T_{k_1, k_2}^{c_{k_1}, c_{k_2}}\ {\rm and}\ u=t_{l_1, l_2}^{d_{l_1}, d_{l_2}}\, \tau_{m_1, m_2}^{e_{m_1}, e_{m_2}}\, T_{n_1, n_2}^{f_{n_1}, f_{n_2}}.
\]

\smallbreak

\noindent Then, we define the following ordering:

\smallbreak

\begin{itemize}
\item[(A)] If $exp(w)\,  <\, exp(u)$, then $w<u$.

\vspace{.1in}

\item[(B)] If $exp(w)\, =\, exp(u)$, then:

\vspace{.1in}

\noindent  (i) if $ind_T(w)\, <\, ind_T(u)$, then $w<u$,

\vspace{.1in}

\noindent  (ii) if $ind_T(w)\, =\, ind_T(u)$, then:

\vspace{.1in}

\noindent \ \ \ \ (a) if $ind_{\tau}(w)\, <\, ind_{\tau}(u)$, then $w<u$,

\vspace{.1in}

\noindent \ \ \ \ (b) if $ind_{\tau}(w)\, =\, ind_{\tau}(u)$, then:

\vspace{.1in}

\noindent \ \ \ \ \ \ \ ($\bullet$) if $ind_{t}(w)\, <\, ind_{t}(u)$, then $w<u$, 

\vspace{.1in}

\noindent \ \ \ \ \ \ \ ($\bullet$) if $ind_{t}(w)\, =\, ind_{t}(u)$, then: 

\vspace{.1in}

\noindent \ \ \ \ \ \ \ \ \ \ \ ($\alpha$) if $i_1=l_1, \ldots, i_{s-1}=l_{s-1}, i_s<l_s$, then $w>u$.

\vspace{.1in}

\noindent \ \ \ \ \ \ \ \ \ \ \ ($\beta$) if $i_q=l_q, \forall q$ and $j_1=m_1, \ldots, j_{p-1}=m_{p-1}, j_p<m_p$, then $w>u$.

\vspace{.1in}

\noindent \ \ \ \ \ \ \ \ \ \ \ ($\gamma$) if $i_q=l_q, \forall q, j_r=m_r, \forall r$ and $k_1=n_1, \ldots, k_{t-1}=n_{t-1}, k_t<n_t$, then $w>u$.

\vspace{.1in}

\noindent \ \ \ \ \ \ \ \ \ \ \ ($\delta$) if $i_q=l_q, j_r=m_r, k_z=n_z \forall q, r, z$, then:

\vspace{.1in}

\noindent \ \ \ \ \ \ \ \ \ \ \ \ \ \ \ \ ($I$) if $c_{k_2}=f_{n_2}, \ldots, c_{k_x}=f_{n_x}, c_{k_x-1}<f_{n_x-1}$, then $w<u$.

\vspace{.1in}

\noindent \ \ \ \ \ \ \ \ \ \ \ \ \ \ \ ($II$) if $c_{k_i}=f_{n_i}, \forall i$, then:

\vspace{.1in}
\begin{itemize}
\item[-] if $b_{j_2}=e_{m_2}, \ldots, b_{j_n}=e_{m_n}, b_{j_{n-1}}<e_{m_{n-1}}$, then $w<u$.
\bigbreak
\item[=] if $b_{j_n}=e_{m_n}, \forall n$, then if $a_{i_2}=d_{l_2}, \ldots, a_{i_v}=d_{l_v}, a_{i_{v-1}}<d_{l_{v-1}}$, then $w<u$.
\bigbreak
\item[$\equiv$] if $a_{i_n}=d_{l_n}, \forall n$, then $w=u$. 
\end{itemize}

\end{itemize}
\end{defn}

The same ordering is defined on the subsets $B_{H_2}, B^{\prime}_{H_2}$ and $\mathcal{B}_{H_2}$ of $L$.

\begin{prop}\label{totord}
The set $L$ equipped with the ordering given in Definition~\ref{order}, is a totally ordered set.
\end{prop}

\begin{proof}
In order to show that the set $L$ is totally ordered set when equipped with the ordering given in Definition~\ref{order}, we need to show that the ordering relation is antisymmetric, transitive and total. We only show that the ordering relation is transitive. Antisymmetric property follows similarly. Totality follows from Definition~\ref{order} since all possible cases have been considered. Let:

\[
\begin{array}{lcl}
w & = & t_{i_1, i_2}^{a_{i_1, i_2}}\, \tau_{i_3, i_4}^{a_{i_3, i_4}}\, T_{i_5, i_6}^{a_{i_5, i_6}}\\
&&\\
u & = & t_{j_1, j_2}^{b_{j_1, j_2}}\, \tau_{j_3, j_4}^{b_{j_3, j_4}}\, T_{j_5, j_6}^{b_{j_5, j_6}}\\
&&\\
v & = & t_{k_1, k_2}^{c_{k_1, k_2}}\, \tau_{k_3, k_4}^{c_{k_3, k_4}}\, T_{k_5, k_6}^{c_{k_5, k_6}}\\
\end{array}
\]

\noindent such that $w<u$ and $u<v$. Then, one of the following holds:

\smallbreak

\begin{itemize}
\item[a.] Either $exp(w)<exp(u)$, and since $u<v$, we have $exp(u)\leq exp(v)$ and thus, $u<v$.
\smallbreak
\item[b.] Either $exp(w)=exp(u)$ and $ind_T(w)<ind_T(u)$. Then, since $u<v$ we have that either $exp(u)<exp(v)$ (same as in case (a)), or $exp(u)=exp(v)$ and $ind_T(u)\leq ind_T(v)$. Thus, $ind_T(w)<ind_T(v)$ and so we conclude that $w<v$.
\smallbreak
\item[c.] Either $exp(w)=exp(u)$, $ind_T(w)=ind_T(u)$ and $ind_{\tau}(w)<ind_{\tau}(u)$. Then, since $u<v$, we have that either:
\smallbreak
\begin{itemize}
\item[$\bullet$] $exp(u)<exp(v)$, same as in case (a), or
\smallbreak
\item[$\bullet$] $exp(u)=exp(v)$ and $ind_T(u)<ind_T(v)$, same as in case (b), or
\smallbreak
\item[$\bullet$] $exp(u)=exp(v)$, $ind_T(u)=ind_T(v)$ and $ind_{\tau}(u)<ind_{\tau}(v)$. Thus, $w<v$.
\end{itemize}
\smallbreak
\item[d.] Either $exp(w)=exp(u)$, $ind_T(w)=ind_T(u)$, $ind_{\tau}(w)=ind_{\tau}(u)$ and $ind_t(w)<ind_t(u)$. Then, since $u<v$, we have that either:
\smallbreak
\begin{itemize}
\item[$\bullet$] $exp(u)<exp(v)$, same as in case (a), or
\smallbreak
\item[$\bullet$] $exp(u)=exp(v)$ and $ind_T(u)<ind_T(v)$, same as in case (b), or
\smallbreak
\item[$\bullet$] $exp(u)=exp(v)$, $ind_T(u)=ind_T(v)$ and $ind_{\tau}(u)<ind_{\tau}(v)$, same as in case (c), or
\smallbreak
\item[$\bullet$] $exp(u)=exp(v)$, $ind_T(u)=ind_T(v)$, $ind_{\tau}(u)=ind_{\tau}(v)$ and $ind_t(u)<ind_t(v)$. Thus, $w<v$.
\end{itemize}
\smallbreak
\item[e.] Either $exp(w)=exp(u)$, $ind_T(w)=ind_T(u)$, $ind_{\tau}(w)=ind_{\tau}(u)$, $ind_t(w)=ind_t(u)$ and $i_1=j_1, \ldots, i_{s-1}=j_{s-1}, i_s>j_s$. Then, since $u<v$, we have that either: 
\smallbreak
\begin{itemize}
\item[$\bullet$] $exp(u)<exp(v)$, same as in case (a), or
\smallbreak
\item[$\bullet$] $exp(u)=exp(v)$ and $ind_T(u)<ind_T(v)$, same as in case (b), or
\smallbreak
\item[$\bullet$] $exp(u)=exp(v)$, $ind_T(u)=ind_T(v)$ and $ind_{\tau}(u)<ind_{\tau}(v)$, same as in case (c), or
\smallbreak
\item[$\bullet$] $exp(u)=exp(v)$, $ind_T(u)=ind_T(v)$, $ind_{\tau}(u)=ind_{\tau}(v)$ and $ind_t(u)=ind_t(v)$, same as in case (d), or
\smallbreak
\item[$\bullet$] $exp(u)=exp(v)$, $ind_T(u)=ind_T(v)$, $ind_{\tau}(u)=ind_{\tau}(v)$, $ind_t(u)=ind_t(v)$ and $j_1=k_1, \ldots, j_{p-1}=k_{p-1}, j_p>k_p$. Then:
\smallbreak
\begin{itemize}
\item[(i)] if $s=p$, then $i_s>j_s>k_s$ and thus $w<v$.
\smallbreak
\item[(ii)] if $s<p$, then $i_s>j_s=k_s$ and thus $w<v$.
\end{itemize}
\end{itemize}
\smallbreak
\item[f.] Either $exp(w)=exp(u)$, $ind_T(w)=ind_T(u)$, $ind_{\tau}(w)=ind_{\tau}(u)$, $ind_t(w)=ind_t(u)$, $i_s=j_s, \forall s$ and $b_{i_6}=a_{i_6}, \ldots, b_{i_x}=a_{i_x}, b_{i_x-1}<a_{i_x-1}$. Then, since $u<v$, we have that either: 
\smallbreak
\begin{itemize}
\item[$\bullet$] $exp(u)<exp(v)$, same as in case (a), or
\smallbreak
\item[$\bullet$] $exp(u)=exp(v)$ and $ind_T(u)<ind_T(v)$, same as in case (b), or
\smallbreak
\item[$\bullet$] $exp(u)=exp(v)$, $ind_T(u)=ind_T(v)$ and $ind_{\tau}(u)<ind_{\tau}(v)$, same as in case (c), or
\smallbreak
\item[$\bullet$] $exp(u)=exp(v)$, $ind_T(u)=ind_T(v)$, $ind_{\tau}(u)=ind_{\tau}(v)$ and $ind_t(u)=ind_t(v)$, same as in case (d), or
\smallbreak
\item[$\bullet$] $exp(u)=exp(v)$, $ind_T(u)=ind_T(v)$, $ind_{\tau}(u)=ind_{\tau}(v)$, $ind_t(u)=ind_t(v)$, and $j_1=k_1, \ldots, j_{p-1}=k_{p-1}, j_p>k_p$, same as in case (e), or
\smallbreak
\item[$\bullet$] $exp(u)=exp(v)$, $ind_T(u)=ind_T(v)$, $ind_{\tau}(u)=ind_{\tau}(v)$, $ind_t(u)=ind_t(v)$, $j_s=k_s, \forall s$ and $c_{k_6}=b_{j_6}, \ldots, c_{k_x}=b_{j_x}, c_{k_x-1}>b_{n_x-1}$. Then, 
\smallbreak
\smallbreak
\begin{itemize}
\item[(i)] if $s=x$, then $a_{i_s-1}<b_{i_s-1}<c_{i_s-1}$ and thus $w<v$.
\smallbreak
\item[(ii)] if $s<x$, then $a_{i_s-1}<b_{i_s-1}=c_{i_s-1}$ and thus $w<v$.
\end{itemize}
\end{itemize}
\end{itemize}
\smallbreak

So, we conclude that the ordering relation is transitive.
\end{proof}

\begin{remark}
Proposition~\ref{totord} also holds for the sets $B_{H_2}$, $B^{\prime}_{H_2}$ and $\mathcal{B}_{H_2}$.
\end{remark}

\begin{prop}\label{totord2}
The sets $B_{H_2}$ $B^{\prime}_{H_2}$ and $\mathcal{B}_{H_2}$ are totally ordered and well-ordered.
\end{prop}

\begin{proof}
Let $B$ denote $B_{H_2}$, $B^{\prime}_{H_2}$ or $\mathcal{B}_{H_2}$. Since $B\subset L$, $B$ inherits the property of being a totally ordered
set from $L$. Moreover, the element $t^{0}\tau^0T^0$ that represents the unknot, is the minimum element of $B$ and so $B$ is a well-ordered set.
\end{proof}

We also introduce the notion of homologous words in the sets $B_{H_2}, B^{\prime}_{H_2}$ and $\mathcal{B}_{H_2}$ as follows:

\begin{defn}\label{hom}\rm
Let $w\in B_{H_2}$. We define the homologous word of $w$ in $B^{\prime}_{H_2}$ to be $w^{\prime}\ =\ t^{l_t}\, {\tau_1^{\prime}}^{l_{\tau}}\, {T_2^{\prime}}^{l_T}$, denoted by $w\overset{BB^{\prime}}{\sim} w^{\prime}$. Similarly, we define the homologous word of $w^{\prime}$ in $\mathcal{B}_{H_2}$ to be $w^{\prime \prime}\ =\ t^{l_t}\, \tau^{l_{\tau}}\, T^{l_T}$, denoted by $w^{\prime}\overset{B^{\prime}\mathcal{B}}{\sim} {w}^{\prime \prime}$. Finally, we denote by $w\overset{B\mathcal{B}}{\sim} {w}^{\prime \prime}$ the homologous word of $w$ in $\mathcal{B}_{H_2}$.
\end{defn}

\subsection{The basis $B^{\prime}_{H_2}$} 

As mentioned in Remark~\ref{sets}, a basic element in the KBSM($H_2$) consists of three different sets of strands depending on the loopings $t^{\prime}_i$'s, $\tau^{\prime}_i$'s and $T^{\prime}_i$'s. Obviously, when the Kauffman bracket skein relation is applied in a set of loopings, the other two sets of loopings remain unaffected. This follows naturally by parting an algebraic mixed braid back to a geometric mixed braid as shown in Figure~\ref{partequi}. Thus, we have the following:

\begin{figure}[!ht]
\begin{center}
\includegraphics[width=4.2in]{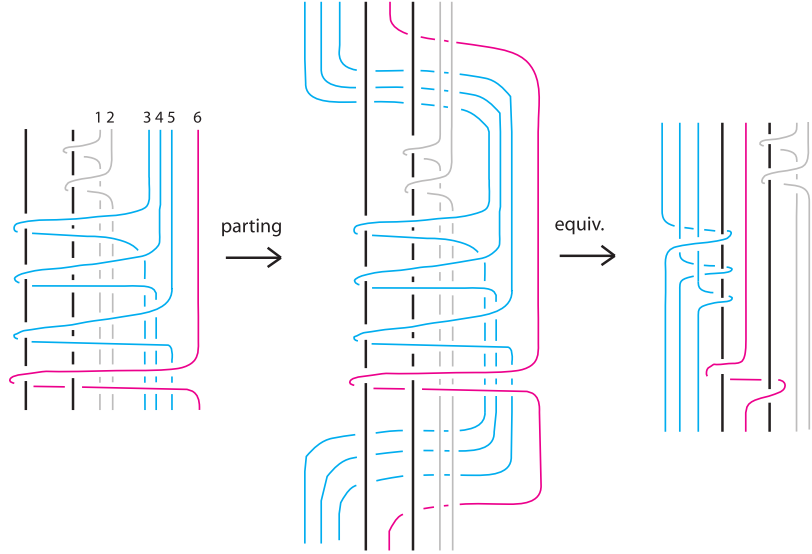}
\end{center}
\caption{ Parting an element in $B_{H_2}$. }
\label{partequi}
\end{figure}

\begin{prop}\label{tauind}
In the Kauffman bracket skein module of $H_2$ a monomial of the form $\alpha\alpha_1^{\prime}\ldots \alpha_n^{\prime}$ can be written as $\alpha^{n+1}$ followed by lower order terms in $B^{\prime}_{H_2}$, where $\alpha$ is either $t$, $\tau$ or $T$.
\end{prop}

\begin{proof}

We prove Proposition~\ref{tauind} by strong induction on the order of the initial monomial. We only prove the case where $\alpha$ is $t$. The other cases follow similarly. Note also that in the figures that follow we omit the scalars that appear after we apply the Kauffman bracket relations.

\smallbreak

Consider the monomial $tt_1^{\prime}\ldots t_n^{\prime}$. The base of the induction is illustrated in Figure~\ref{tau1}, where we start from the monomial $tt_1^{\prime}\in B_{H_2}$ and applying the Kauffman bracket skein relation we obtain the $BB^{\prime}$-homologous word $t^2\in B^{\prime}_{H_2}$ and the unknot, which is of lower order than $t^2$.

\begin{figure}[!ht]
\begin{center}
\includegraphics[width=5.6in]{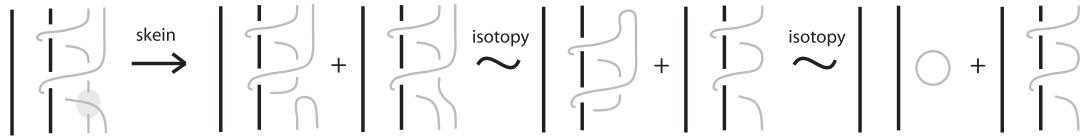}
\end{center}
\caption{ Base of the induction of Prop.~\ref{tauind}. }
\label{tau1}
\end{figure}

We assume now that Proposition~\ref{tauind} holds for all monomials in $t$'s of less order than $tt_1^{\prime}\ldots t_n^{\prime}$. Then for $tt_1^{\prime}\ldots t_n^{\prime}$ we have that after applying the Kauffman bracket skein relation we obtain the monomials $tt_1^{\prime}\ldots t_{n-2}^{\prime}\in B^{\prime}_{H_2}$ and $tt_1^{\prime}\ldots {t_{n-1}^{\prime}}^2\in L$, which are of lower order than the initial monomial $tt_1^{\prime}\ldots t_n^{\prime}\in B_{H_2}$, and thus, Proposition~\ref{tauind} is true for all monomials in $t$'s (see Figure~\ref{tau2}).

\begin{figure}[!ht]
\begin{center}
\includegraphics[width=5.6in]{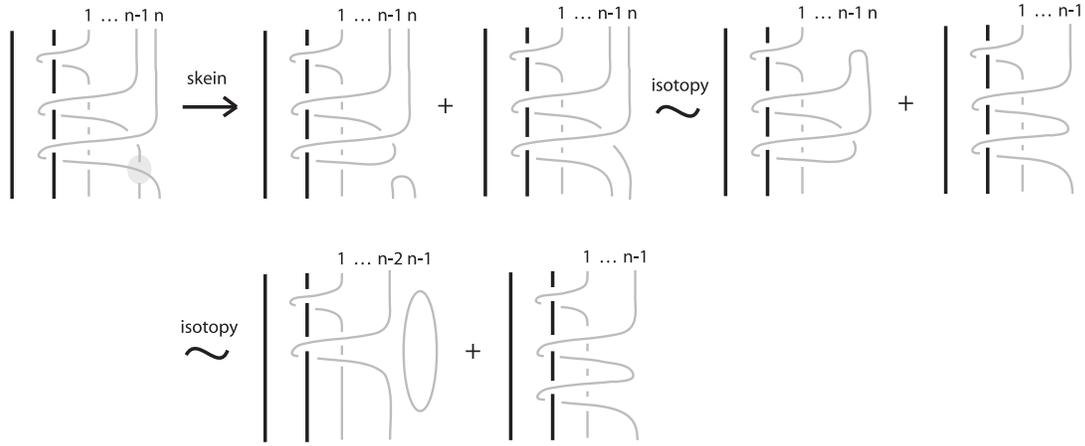}
\end{center}
\caption{ Proof of Proposition~\ref{tauind}. }
\label{tau2}
\end{figure}

\end{proof}

\begin{remark}\rm
Note that in the proof of Proposition~\ref{tauind}, when we apply the Kauffman bracket skein relation on $tt_1^{\prime}\ldots t_n^{\prime}$, we obtain the monomial $tt_1^{\prime}\ldots {t_{n-1}^{\prime}}^2\in L\backslash B_{H_2}$, but applying the skein relation on all crossings, this monomial is expressed as a combination of elements in $B^{\prime}_{H_2}$ of lower order than the $BB^{\prime}$-homologous word.
\end{remark}

In order to prove Theorem~\ref{newbasis1}, we now part again the resulting geometric mixed braids as illustrated in Figure~\ref{bas1} and the proof is concluded.

\begin{figure}[!ht]
\begin{center}
\includegraphics[width=2.8in]{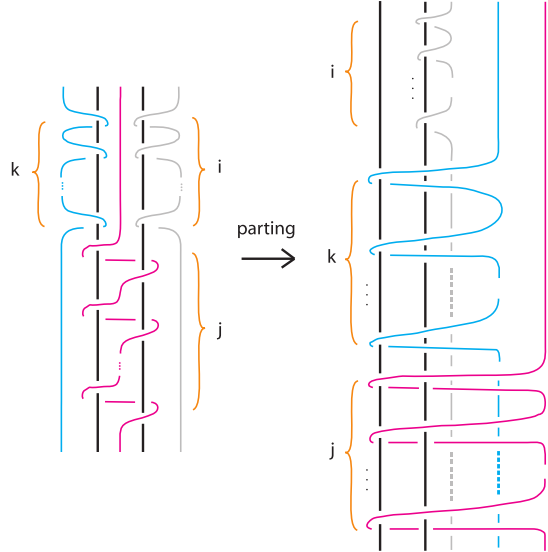}
\end{center}
\caption{ The new basis $B^{\prime}_{H_2}$. }
\label{bas1}
\end{figure}

\subsection{The basis $\mathcal{B}_{H_2}$} 

We now pass from the basis $B^{\prime}_{H_2}$ that consists of monomials in $t$'s, $\tau_1^{\prime}$'s and $T_2^{\prime}$'s, to elements in the set $\mathcal{B}_{H_2}$ that consists of monomials in $t$, $\tau$ and $T$. This set forms a more natural basis for KBSM($H_2$) on the braid level, since there are no braiding crossings. We prove Theorem~\ref{newbasis2} by showing that an element $w$ in the basis $B^{\prime}_{H_2}$ can be expressed as a sum of the $B^{\prime}\mathcal{B}$-homologous element $w^{\prime \prime}$ and elements in $\mathcal{B}_{H_2}$ of lower order than $w^{\prime \prime}$. We prove that by strong induction on the order of $w\in B^{\prime}_{H_2}$.

\smallbreak

\begin{proof}
The base of the induction is shown in Figure~\ref{tbasis} where we start from the element $t\tau_1^{\prime}$ and applying the Kauffman bracket skein relation we obtain the $B\mathcal{B}$-homologous element $t\tau$ and $t<t\tau$. See Figure~\ref{tbasis22} for an example of the same procedure applied on the elements $tT_1^{\prime}$ and $\tau T_1^{\prime}$.

\begin{figure}[!ht]
\begin{center}
\includegraphics[width=4.7in]{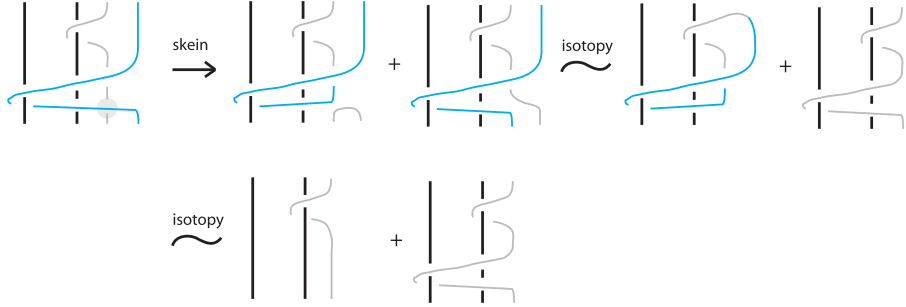}
\end{center}
\caption{ The base of the induction for Theorem~\ref{newbasis2}. }
\label{tbasis}
\end{figure}

\begin{figure}[!ht]
\begin{center}
\includegraphics[width=4.7in]{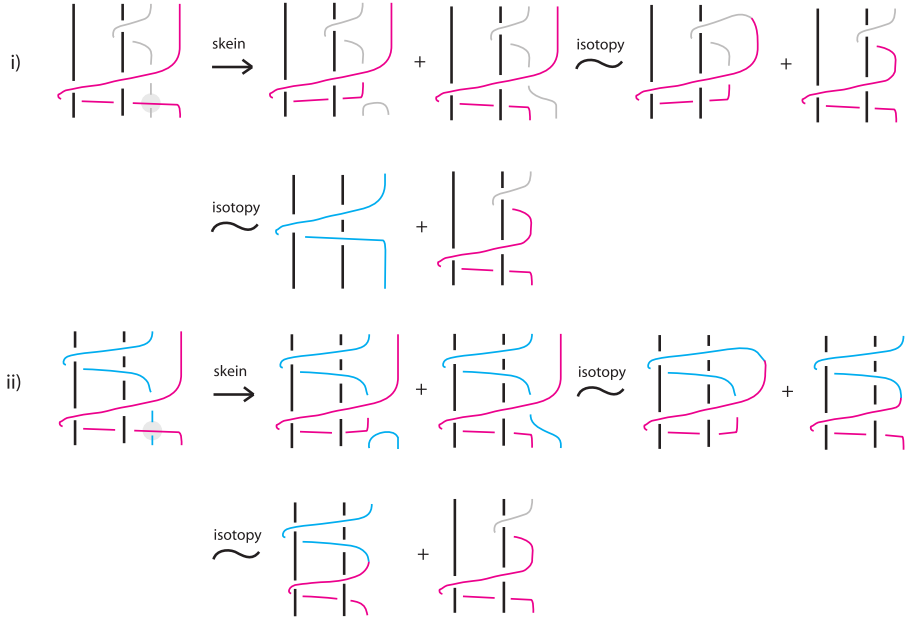}
\end{center}
\caption{ Converting $tT_1^{\prime}$ and $\tau T_1^{\prime}$ to elements in $\mathcal{B}_{H_2}$.}
\label{tbasis22}
\end{figure}

We now assume that the statement holds for all monomials of less order than $w=t^i{\tau_1^{\prime}}^k {T_2^{\prime}}^j\in B^{\prime}_{H_2}$. Then, for $t^i{\tau_1^{\prime}}^k {T_2^{\prime}}^j$ we have that after applying the Kauffman bracket skein relation, $t^i{\tau_1^{\prime}}^k {T_2^{\prime}}^j$ is written as a sum of elements in $L$ of lower order than $w$ and in particular, $t^i{\tau_1^{\prime}}^{k-1}\tau_2^{\prime}{T_2^{\prime}}^{j-1}$ and $t^i{\tau_1^{\prime}}^{k}{T_1^{\prime}}^{j}$ (see Figure~\ref{tbasis2}). Thus, by the induction step, the statement holds for $w$.

\begin{figure}[!ht]
\begin{center}
\includegraphics[width=5.5in]{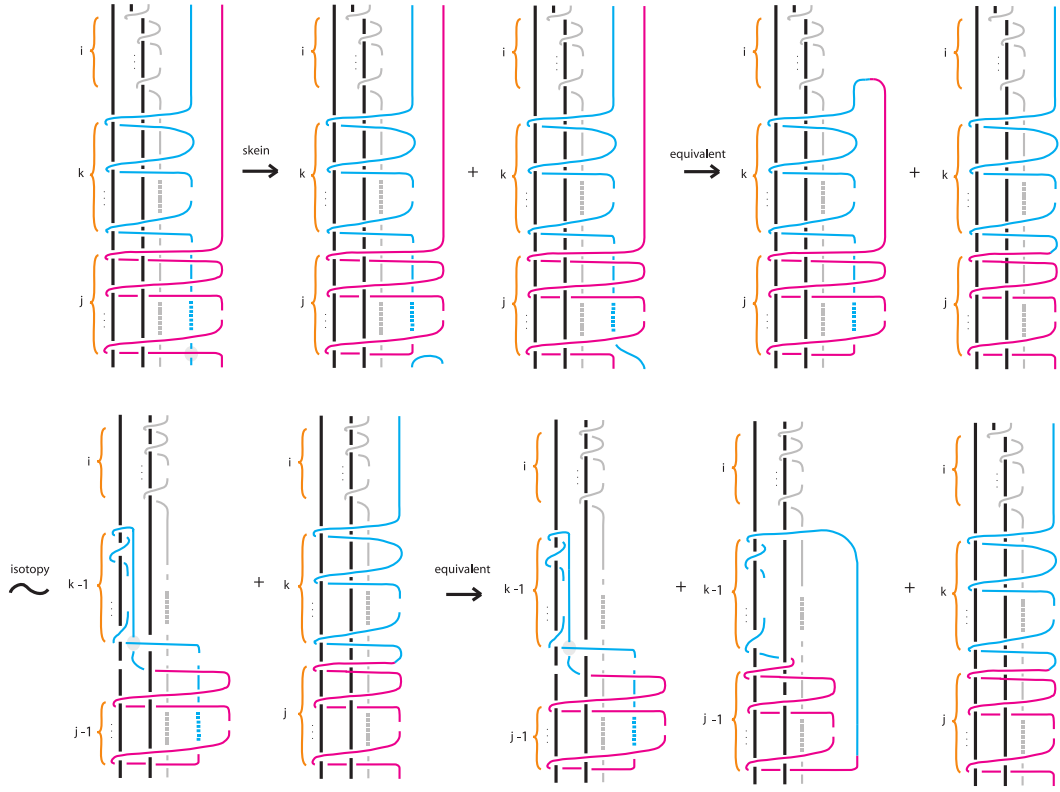}
\end{center}
\caption{ Proof of Theorem~\ref{newbasis2}. }
\label{tbasis2}
\end{figure}

\end{proof}

\begin{figure}[!ht]
\begin{center}
\includegraphics[width=4in]{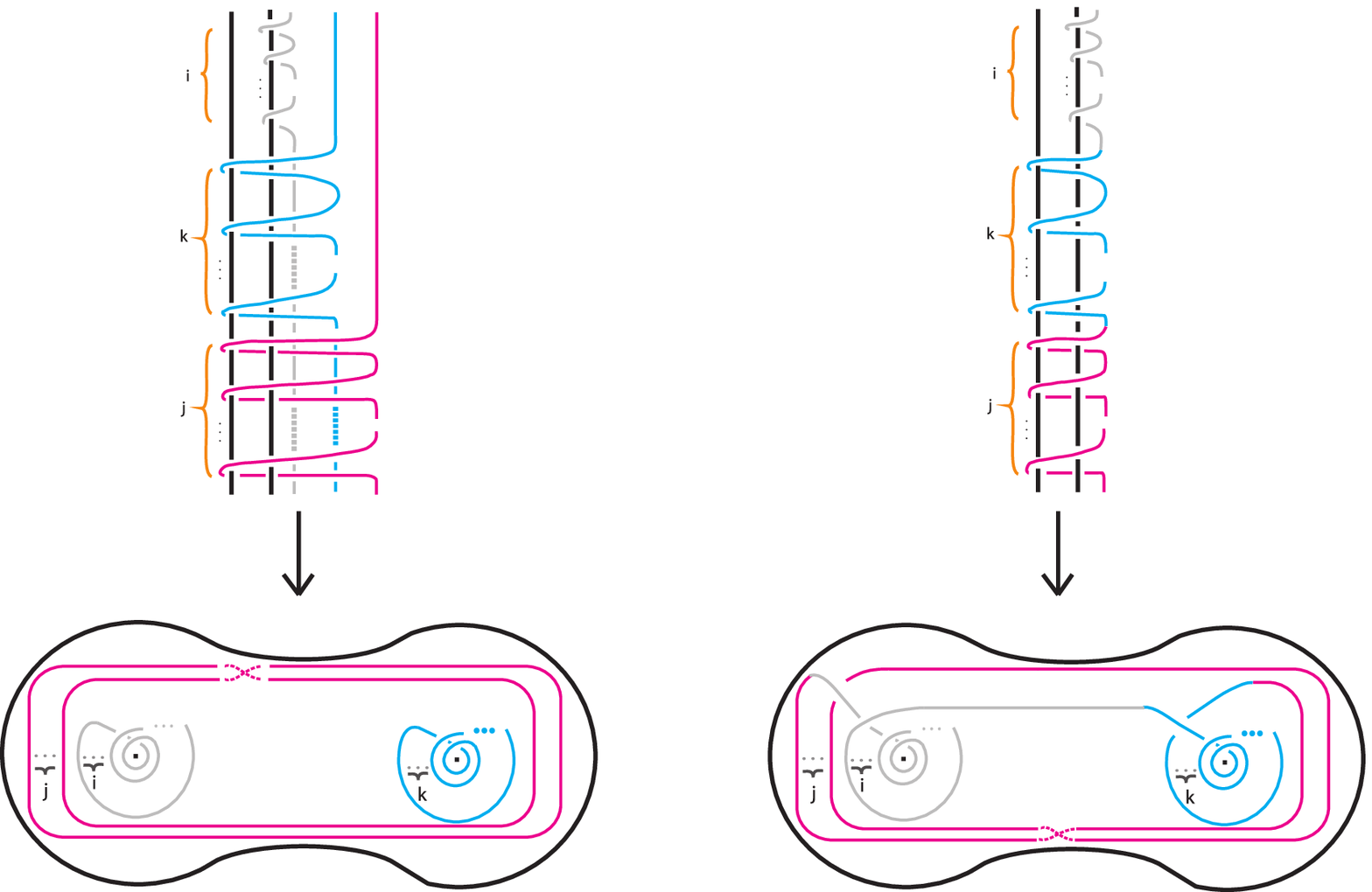}
\end{center}
\caption{The sets $B^{\prime}_{H_2}$ and $\mathcal{B}_{H_2}$.}
\label{basesall}
\end{figure}

\section{Conclusions}

In this paper we present two different bases for the Kauffman bracket skein module of the handlebody of genus 2 via braids. On the level of braids the basis $\mathcal{B}_{H_2}$ is more natural since it consists of elements with no moving crossings (see Figure~\ref{basesall}). As explained in \cite{D, DL1, DL2, DL3, DL4}, the braid band moves that reflect isotopy in a closed-connected-oriented (c.c.o.) 3-manifold $M$ obtained by $H_2$ by surgery, are naturally described with the use of the new basis $\mathcal{B}_{H_2}$. The idea is that using this basis we will compute the Kauffman bracket skein module of $M$. Moreover, using the basis $\mathcal{B}_{H_2}$, in \cite{D2} we work on the computation of the Kauffman bracket skein module of the complement of the trefoil knot via braids.


\begin{thebibliography}{ABCD}

\bibitem[D]{D} {\sc I. Diamantis}, An Alternative Basis for the Kauffman Bracket Skein Module of the Solid Torus via Braids, (2019) In: Adams C. et al. (eds) Knots, Low-Dimensional Topology and Applications. KNOTS16 2016. Springer Proceedings in Mathematics \& Statistics, vol 284. Springer, Cham.

\bibitem[D2]{D2} {\sc I. Diamantis}, On the Kauffman bracket skein module of the complement of the trefoil knot via braids, work in progress.

\bibitem[DL1]{DL1} {\sc I. Diamantis, S. Lambropoulou}, Braid equivalences in 3-manifolds 
with rational surgery description, {\em Topology and its Applications}, {\bf 194} (2015), 269-295.

\bibitem[DL2]{DL2} {\sc I. Diamantis, S. Lambropoulou}, A new basis for the HOMFLYPT skein module of the solid torus, {\em J. Pure Appl. Algebra} {\bf 220} Vol. 2 (2016), 577-605.

\bibitem[DL3]{DL3} {\sc I. Diamantis, S. Lambropoulou}, The braid approach to the HOMFLYPT skein module of the lens spaces $L(p, 1)$, Springer Proceedings in Mathematics and Statistics (PROMS),{\em Algebraic Modeling of Topological and Computational Structures and Application}, (2017).

\bibitem[DL4]{DL4} {\sc I. Diamantis, S. Lambropoulou}, An important step for the computation of the HOMFLYPT skein module of the lens spaces $L(p,1)$ via braids, arXiv:1802.09376v1[math.GT], to appear in J. Knot Theory Ramif., special issue dedicated to the Proceedings of the International Conference on Knots, Low-dimensional Topology and Applications - Knots in Hellas 2016.

\bibitem[DLP]{DLP} {\sc I. Diamantis, S. Lambropoulou, J. H. Przytycki}, Topological steps on the HOMFLYPT skein module of the lens spaces $L(p,1)$ via braids, {\it J. Knot Theory and Ramifications}, {\bf 25}, No. 14, (2016).

\bibitem[GM]{GM} {\sc B. Gabrov\v sek, M. Mroczkowski},  Link diagrams and applications to skein modules, {\em Algebraic Modeling of Topological and Computational Structures and Applications}, Springer Proceedings in Mathematics \& Statistics (2017).

\bibitem[KL]{KL}{\sc D. Kodokostas, S. Lambropoulou} A spanning set and potential basis of the mixed Hecke algebra on two fixed strands, {\it Mediterr. J. Math.} (2018), 15:192, https://doi.org/10.1007/s00009-018-1240-7.

\bibitem[La1]{La1}{\sc S. Lambropoulou}, Braid structures in handlebodies, knot complements and 3-manifolds,  {\it Proceedings of Knots in Hellas '98}, World Scientific Press, Series of Knots and Everything {\bf 24}, (2000) 274-289.

\bibitem[LR1]{LR1} {\sc S. Lambropoulou, C.P. Rourke} (2006), Markov's theorem in $3$-manifolds, \emph{Topology and its Applications} {\bf 78},
(1997) 95-122.

\bibitem[LR2]{LR2} {\sc S. Lambropoulou, C. P. Rourke}, Algebraic Markov equivalence for links in $3$-manifolds, {\em Compositio Math.} {\bf 142} (2006) 1039-1062.

\bibitem[OL]{OL} {\sc Reinhard H{\"a}ring-Oldenburg, Sofia Lambropoulou}, Knot theory in handlebodies, {\it J. Knot Theory and its Ramifications} {\bf 11}, No. 6, (2002) 921-943.

\bibitem[P]{P} {\sc J.~Przytycki}, Skein modules of 3-manifolds, {\it Bull. Pol. Acad. Sci.: Math.}, {\bf 39, 1-2} (1991), 91-100.

\bibitem[Tu]{Tu} {\sc V.G.~Turaev}, The Conway and Kauffman modules of the solid torus,  {\it Zap. Nauchn. Sem. Lomi} {\bf 167} (1988), 79--89. English translation: {\it J. Soviet Math.} (1990), 2799-2805.
\end{thebibliography}
\end{document}